\documentclass[a4paper,oneside,english,reqno,12pt]{amsart}
\usepackage[utf8]{inputenc}
\usepackage[T1]{fontenc}
\usepackage{graphicx}
\DeclareFontFamily{OMX}{mlmex}{}
\DeclareFontShape{OMX}{mlmex}{m}{n}{%
   <->mlmex10%
   }{}%
\usepackage{mlmodern}

\usepackage[hscale=0.666, vscale=0.72]{geometry}
\usepackage{babel}
\usepackage[numbers,sort&compress]{natbib}

\usepackage[np,autolanguage]{numprint}
\usepackage{hyperref}
\hypersetup{%
  colorlinks=true,%
  citecolor=[RGB]{120,29,126},%
  pdfauthor={Jean-François Burnol},%
  pdfsubject={Block-count constrained harmonic sums},%
  pdfstartview=FitH,%
  pdfpagemode=UseNone,%
}

\usepackage{amsmath,amsthm,amssymb}
\usepackage{mathtools}

\DeclarePairedDelimiterX\Iffint[2]{\lbrack\!\lbrack}{\rbrack\!\rbrack}{#1\dots#2}
\DeclarePairedDelimiterX\Ioo[2]{\lparen}{\rparen}{#1,#2}
\DeclarePairedDelimiterX\Iof[2]{\lparen}{\rbrack}{#1,#2}
\DeclarePairedDelimiterX\Ifo[2]{\lbrack}{\rparen}{#1,#2}
\DeclarePairedDelimiterX\Iff[2]{\lbrack}{\rbrack}{#1,#2}

\DeclarePairedDelimiterX\scalp[2]{\langle}{\rangle}{#1\mid#2}

\newcommand\emptyword{\epsilon}

\newcommand\ZZ{\mathbb{Z}}

\newcommand\Un{\mathbf{1}}

\DeclareMathOperator\Leb{Leb}

\DeclareMathOperator\Cyl{Cyl}

\DeclareMathOperator\cB{\mathcal{B}}
\DeclareMathOperator\cC{\mathcal{C}}

\DeclareMathOperator\cF{\mathcal{F}}
\DeclareMathOperator\cG{\mathcal{G}}

\DeclareMathOperator\cH{\mathcal{H}}

\DeclareMathOperator\cL{\mathcal{L}}
\DeclareMathOperator\cR{\mathcal{R}}

\DeclareMathOperator\cW{\mathcal{W}}

\DeclareMathOperator\wH{\widehat{H}}
\DeclareMathOperator\wM{\widehat{M}}

\newcommand\wmu{\widehat{\mu}}
\newcommand\weta{\widehat{\eta}}

\DeclareMathOperator\sD{\mathsf{D}}
\DeclareMathOperator\sE{\mathsf{E}}

\DeclareMathOperator\sK{\mathsf{K}}

\DeclareMathOperator\sT{\mathsf{T}}

\newcommand\dmu{\mathrm{d}\mu}
\newcommand\deta{\mathrm{d}\eta}
\newcommand\dnu{\mathrm{d}\nu}

\newcommand\rM{\mathrm{M}}
\newcommand\rH{\mathrm{H}}

\newcommand\fL{\mathfrak{L}}
\newcommand\fR{\mathfrak{R}}

\theoremstyle{plain}
\newtheorem{theo}{Theorem}
\newtheorem{prop}{Proposition}

\theoremstyle{definition}

\allowdisplaybreaks

\title{The modal expansion of Kempner sums}

\author[J.-F. Burnol]{Jean-François Burnol}

\date{25 September 2026.}


\usepackage{setspace}
\newcommand\arxivurl[1]{\href{https://arxiv.org/abs/#1}{\textsf{arXiv:#1}}}

\begin{document}
\addtocontents{toc}{\protect\hypersetup{hidelinks}}

\begin{abstract}
  This is a continuation of the author's earlier paper \emph{Block-count
    constrained harmonic sums: spectral expansion and block-directed
    Euler--Maclaurin}.  We presented in that
  reference the \emph{modal expansion} of block-Irwin sums $I(b,w,k)$ for
  $k\geq1$.  The case $k=0$ was fully developed when $w$ is a single digit,
  but it was presented only as a very short result summary for $w$ a multi-digit
  block.  Although the framework is already completely exposed in the earlier
  paper, the case $k=0$ requires adding a separate, dedicated discussion, which is
  provided here.
\end{abstract}

\maketitle

\tableofcontents

\onehalfspacing

\setcounter{part}{3}
\setcounter{section}{35}
\setcounter{equation}{502}
\setcounter{theo}{4}
\setcounter{prop}{21}

\part{The case \texorpdfstring{$k=0$}{k=0}}

\section{Introduction}

This paper is a direct continuation to \cite[Part 2]{burnolmodal}.  It
provides (in its last four sections) the detailed explanations for the results
summarized in a very condensed manner there in section~32, and, prior
to that, quite more.  The rationale for including only a summary in
\cite{burnolmodal} resides in the size which was already attained at the end
of its Part~2, and in the author's resolve to have enough room for a fully
developed Part~3 on the block-directed Euler--Maclaurin idea.

I thank Thomas Schmelzer for sending me a manuscript \cite{schmelzer2026},
which had been completed at about the same time as my own \cite{burnolmodal}.
There is some common material, but the respective aims are distinct enough,
and the points of overlap few enough, that a detailed comparison is not
needed here: roughly, Schmelzer's paper is concerned with uniform estimates
and asymptotic classification, and uses the prefix automaton in an essential
way, whereas the perspective of both \cite{burnolmodal} and the present paper
is principally that of the modal expansion. The prefix automaton
is described in \cite{burnolmodal}, but is not used to derive the modal
expansion.  In \cite[Part~1]{burnolmodal}, we did
devote some space to inequalities but only in the $|w|=1$ case.

There is one overlap I need to mention here. After finishing this paper I read
Schmelzer's more closely and recognized in his Lemma~5(a) the same statement
about bordered words as we will establish below: if a bordered word has length
$p$ and shortest positive period $p_0$, every overlap period not exceeding
$p-p_0$ is a multiple of $p_0$. This result is an easy consequence of the GCD
rule mentioned in \cite{guibasodlyzko1981a, guibasodlyzko1981b}.  Here it
forms part of a fuller discussion of the occurrence-free language partition
one encounters (tacitly) in \cite{guibasodlyzko1981a, guibasodlyzko1981b}.  We
give more details than strictly necessary for our other goals, mainly to
compensate for an error in an argument in this area, which was made by the
author in \cite{burnolblocks}.

\begin{footnotesize}\singlespacing
  Here is a
  list of currently known typographical issues in \arxivurl{2609.17045v1}
  \begin{enumerate}
  \item In Equation (120), $a \in igma_b$ should read $a\in \Sigma_b$,
  \item In the paragraph introduction Equation (210), a reference is made
    to \cite[\S6.4]{odlyzko1995}.  It should have been \cite[Example~6.4]{odlyzko1995}.
  \item In Equation (220), the notation $\ell^-$ was left undefined and should
    have been $\ell[{:}{-}1]$, as explained in the paragraph preceding that
    equation.  Furthermore, in that paragraph the \TeX{} mark-up was faulty
    and $g[:-1]$ should have rendered as $g[{:}{-}1]$,
  \item In Part~1, stopping times $\tau_k$ are defined as the location of the
    $(k+1)$-st occurrence of the digit $d$.  But in Part~2, in
    Equations~(271)--(278), there is a shift by one in the index, $\tau_k$ is
    now the location of the last digit of $w$ on its $k$-th occurrence. So the
    $\tau_k$ of Part~2 is the $\tau_{k-1}$ from Part~1.    Besides, the notation
    $\tau_w$ is used for $|w|$ if $w$ is unbordered and for its shortest
    positive period if it is bordered, hence a clash of notation with stopping
    times from Part~1.
  \item The sentence after Equation (285) starting with ``Equivalently,
    $D^m\nu_{m+1}$'' leaves it to the reader to remember, or infer from
    Equation~(285), that the symbol $D$ is the distributional differentiation,
    which was defined as such long ago in Part~1.  And Part~2 uses
    frequently the letter $D$ also for a certain polynomial.
  \end{enumerate}
\end{footnotesize}
\par

For notation not redefined here, refer to \cite{burnolmodal}.  Matters of
words and their borders will be reviewed in depth in section~\ref{sec:Fpartition}.
Recall that for $g$ a word with letters in the alphabet of digits $\Sigma_b$,
$x(g) = n(g)/b^{|g|}$, with $|g|$ the length of $g$, and $n(g)$ the
non-negative integer represented by $g$, with most significant digits on the
left in $g$.

\section{Review of the modal expansion for block-Irwin sums}

Part~2 contains in particular two separate, but related developments:
\begin{itemize}
\item The complete generalization of the earlier $p=1$ work \cite{burnolirwin}
  to the general multi-digit case. The numerical computation of $I(b,w,k)$ for
  any $k\geq0$ is achieved via series which can be chosen to have (within
  practical limits, as there is a trade-off with the number of series used)
  arbitrarily fast geometric convergence.  This follows from the
  ``level-raising'' method, combined with triangular linear schemes for the
  computation of the moments of auxiliary (discrete) measures $\eta_j$, $j\leq
  k$.  The block-Irwin sums themselves are directly related to discrete
  measures $\mu_k$, $k\geq0$, defined in \cite{burnolirwin, burnolblocks}.
\item The \emph{modal expansion} from \cite[Th.\@ 3]{burnolmodal} for
  $k\geq1$:
\begin{equation}\label{eq:modalIrwin}
  I(b,w,k) = \sum_{j=1}^\infty \fL_j(b,w)\lambda_j(b,w)^{k-1}\fR_j(b,w).
\end{equation}
Its roots are in the language
  factorization $\cW_k=\cL \cG^{k-1}\cR$. Here:
\begin{itemize}
\item $\cW_k$ is the language of words with exactly $k$
  occurrences of $w$,
\item $\cL$ is the \emph{initial language}, i.e.\@ the
  language of words containing exactly one occurrence of the word $w$, which is
  terminal,
\item $\cR$ is the \emph{tail language}, which contains the words
  which, if appended immediately after a terminal occurrence of $w$,
  do not add any new occurrence,
\item and $\cG$ is the \emph{next-return} or \emph{renewal} language, which
  contains the words $g$, such that $wg$ contains exactly two occurrences of
  $w$, the second one being terminal, and $\lambda_j(b,w) = \sum_{g\in \cG}
  b^{-j |g|}$.
\end{itemize}
\end{itemize}
The link between the two is that,
defining the \emph{return operator} $K_w$ acting on Borel measures (or even
distributions) on $\Iff01$ via
\begin{equation}
  K_w(\nu) = \sum_{g\in \cG} b^{-|g|}(\phi_g)_* \nu,
\end{equation}
where $\phi_g(t) = x(g) + b^{-|g|}t$ is the affine map associated with the word $g$, one has the renewal identity:
\begin{equation}
  \eta_{k+1} = K_w(\eta_k).
\end{equation}
The restriction of $\phi_g$ to the half-open interval $\Ifo01$ establishes a
bijection with the cylinder $\Cyl_g$ defined by $g$.

The quantities $\lambda_j(b,w)$, abridged from here on to $\lambda_j(w)$, are
the eigenvalues of $K_w$. But measures are not enough: the eigenvectors for
$j>1$ are distributional derivatives of singular measures.  Probability
measures $\nu_{j}$ are defined on $\Iff01$, with $\nu_1=\Leb$, and are such
that the distributional derivatives
$\sigma_{j}=(-1)^{j-1}\nu_{j}^{(j-1)}/(j-1)!$ are the (normalized)
eigenvectors of $K_w$.  Transposing the operator to act on functions, one
obtains:
\begin{equation}
  (\sK_w f)(x) = \sum_{g\in\cG} b^{-|g|}f(\phi_g(x)).
\end{equation}
The action is triangular on the polynomial algebra and has monic
eigenpolynomials $\cB_{m}$, $m\geq0$, there, with eigenvalue
$\lambda_{m+1}(b,w)$, and which are biorthogonal to the eigendistributions.
These two return operators $K_w$ and $\sK_w$ act very
differently when applied both to a continuous function $\phi$ on
$\Iff01$: $K_w$ places a copy of $\phi$ on each cylinder $\Cyl_g$, $g\in\cG$,
parameterized by $\Ifo01$ via $\phi_g$.

There holds
\begin{equation}
  \lambda_j(w) = G(b^{-j})\qquad
  G(t) = \frac{E(t)}{D(t)} = \frac{(1-bt)(C(t)-1)+t^p}{(1 -bt)C(t) + t^p}\;,
\end{equation}
where $C(t)$ is the Guibas-Odlyzko auto-correlation polynomial
\cite{guibasodlyzko1981b}.  With our notation
\begin{equation}\label{eq:Ct}
  C(t) = 1 + \sum_{c\in \cC_w^+} t^{|c|}\,,
\end{equation}
and $\cC_w^+$ is the set of \emph{complements} of $w$, i.e.\@ the complements in
$w$ of its borders.  We define again in the sequel all the needed material on
words and their borders. Words of length $1$ have no borders.

Only the products $\fL_j(b,w)\fR_j(b,w)$ are uniquely determined by the
$I(b,w,k)$ values, $k\geq1$.  The individual factors originate in the method
of proof, which goes via the block-directed Euler--Maclaurin expansion
associated with $w$ (\cite[Th.\@ 4]{burnolmodal}), as applied to the
meromorphic function
\begin{equation}
  \label{eq:fL}
  f_{\cL}(z) = \sum_{\substack{g\in \cL\\g_1>0}}\frac1{n(g) + z}\;.
\end{equation}
See \cite[Eq.\@ 335]{burnolmodal}; $g_1$ is the first digit of $g$.

The application to $f_{\cL}$ of the block-directed Euler--Maclaurin expansion reads
\begin{equation}
  \label{eq:block-EML}
    f_{\cL}(z) = \sum_{m=0}^\infty
   \Bigl((m!)^{-1}\int_0^1 f_{\cL}^{(m)}(t)\dnu_{m+1}(t)\Bigr) \cB_m(z)\;.
\end{equation}
The left modal coefficients $\fL_j(b,w)$ are the above coefficients
$\sigma_j(f_{\cL})$, $j=m+1$.  The right modal coefficients are obtained by
integrating the eigenpolynomials against the measure $\eta_0$:
\begin{equation}
  \fR_j(b,w)=\int_{\Ifo01}\cB_{j-1}(x)\deta_0(x).
\end{equation}
The measure $\eta_0$ is directly
associated with the tail language: $\eta_0 = \sum_{g\in \cR}
b^{-|g|}\delta_{x(g)}$. One has $\fR_1(b,w)=\eta_0(\Iff01)=b^p$.

\section{A modal expansion for block-Kempner sums}

For $k=0$, the language of interest is $\cF$, the occurrence-free language, which
is directly associated with the measure $\mu_0 = \sum_{g\in \cF}
b^{-|g|}\delta_{x(g)}$. 
Keeping close with the original approach in the one-digit case in \cite{burnolkempner}, we first replace the expression
\begin{equation}
  I(b,w,0) = \int_{\Ifo{b^{-1}}{1}}\frac{\dmu_0(x)}{x}\;,
\end{equation}
with an integration on the full (half-open, but this is done only if in future
we extend the measure beyond their original home in $\Ifo01$) interval:
\begin{equation}\label{eq:Ibw0}
  I(b,w,0)=\sum_{a=1}^{b-1}\int_{\Ifo01}\frac{\dmu_0(x)}{a+x}
  -\Un_{\{w_1\neq0\}}(w)\int_{\Ifo01}\frac{\deta_0(x)}{n(w)+x}.
\end{equation}
This follows from Equation~(358) of Part~2, taken with $y=0$.

We want to apply the block-directed Euler--Maclaurin to the involved kernels,
but there arises an issue with $1/(1+x)$ having a pole at $-1$. This is a
problem only if $w=0^p$. Indeed, recall from Part~2 that all poles should be
outside the smallest closed disk centered at $x_w$ and containing $\Iff01$.  Here,
$x_w=0.g_wg_wg_w\dots$ is the fixed-point of the affine map
$\phi_{g_w}$, where $g_w$ is, if $w$ is unbordered, $w$ itself, and, if
not, the shortest complement of a border (its length is the smallest overlap
period of $w$).  And $x_w=0$ happens  exactly when $w=0^p$.  We leave
this case aside for now.

Recall $A_j(n)=\int_0^1(n+x)^{-j}\dnu_j(x)$ and
$\fR_j(b,w)=\int_{\Ifo01}\cB_{j-1}(x)\deta_0(x)$, $j=m+1$, as in Part~2.
\begin{theo}[The modal expansion of block-Kempner sums]
  Let $w\neq 0^p$ be a word of length $p$. There holds
\begin{equation}\label{eq:modalKempner}
  \begin{split}
  I(b,w,0)=\sum_{m=0}^\infty(-1)^m
    \Bigl(&\sum_{a=1}^{b-1}A_{m+1}(a)\int_{\Ifo01}\cB_m(x)\dmu_0(x)
    \\&-\Un_{\{w_1\neq0\}}(w)A_{m+1}(n(w))\fR_{m+1}(b,w)\Bigr).
  \end{split}
\end{equation}
The series is geometrically convergent.
\end{theo}
\begin{proof}
  This follows from Theorem~4 of \cite{burnolmodal} applied to the kernels
  $1/(n+x)$ arising in Equation~\eqref{eq:Ibw0}.  The series is geometrically
  convergent, because the kernel expansions each converge uniformly
  geometrically on $\Iff01$.
\end{proof}

Denote $\beta_{m+1}(b,w)$, $m\geq0$, the terms of this series.
The computation of the $A_{m+1}(n)$ and of the $\fR_{m+1}(b,w)$ is already
provided by Part~2.  We shall explain later how to handle numerically
the computation of
\begin{equation}
  \int_{\Ifo01}\cB_m(x)\dmu_0(x), 
\end{equation}
for $m\geq0$.  It is known from \cite{burnolblocks} (as a corollary of earlier
work by Guibas and Odlyzko \cite{guibasodlyzko1981b} on string overlaps and
the occurrence-free language) that the total mass of $\mu_0$ is
$b^pC(b^{-1})$.  So, as $\cB_0=1$ and $\nu_1=\Leb$, the first modal term is
\begin{equation}
  \beta_1(b,w)=b^pC(b^{-1})\log(b)
  -\Un_{\{w_1\neq0\}}(w)b^p\log\Bigl(1+\frac1{n(w)}\Bigr).
\end{equation}
It is exactly $b^pC(b^{-1})\log(b)$ for a word beginning with zero,
apart from the case $0^p$ which was provisionally set aside.

For $w=0^p$, we apply to Equation~\eqref{eq:Ibw0} one further level raising to
all the kernels $1/(a+x)$, $1\leq a<b$, obtaining new kernels $1/(ba+d+x)$,
$1\leq a<b$, $0\leq d<b$, together with the subtracted kernels
$1/(ab^p+x)$. All are safe for the block-directed Euler--Maclaurin and we obtain
\begin{prop}[The modal expansion for $w=0^p$]
For any $p\geq1$, there holds
  \begin{equation}\label{eq:modalKempnerzero}
    \begin{split}
      I(b,0^p,0)=\sum_{a=1}^{b-1}\frac1a+\sum_{m=0}^\infty(-1)^m
      \Bigl(&\sum_{a=b}^{b^2-1}A_{m+1}(a)\int_{\Ifo01}\cB_m(x)\dmu_0(x)
      \\&-\sum_{a=1}^{b-1}A_{m+1}(ab^p)\fR_{m+1}(b,0^p)\Bigr),
    \end{split}
  \end{equation}
with a geometrically convergent series.
\end{prop}
We denote the terms of the modal series by
$\beta_{m+1}^{(2)}(b,0^p)$, keeping the finite sum outside the modal series.
The superscript $(2)$ is to indicate that the kernels use numbers with at
least two digits in base $b$.  In particular,
\begin{equation}
  \beta_1^{(2)}(b,0^p)=b^pC(b^{-1})\log(b)
  -b^p\sum_{a=1}^{b-1}\log\Bigl(1+\frac1{ab^p}\Bigr).
\end{equation}

We could have done this
level-raising also for $w\neq 0^p$, which would have given (if $|w|>1$) in
particular the finite harmonic sum $\sum_{1\leq a <b} a^{-1}$ and then a modal
expansion, whose terms we would then denote $\beta_{m+1}^{(2)}(b,w)$.
We see that there is a very strong analogy with computing an integral by some
quadrature formula and then proceeding with a (type of) Euler--Maclaurin
asymptotic expansion.  But in our context, this is not only asymptotic but
perfectly well geometrically convergent.

Let us now discuss the contents to follow:
\begin{enumerate}
\item We start with a detailed discussion of the occurrence-free language and
  its Guibas-Odlyzko partition as some aspects will be needed.  There is a
  further reason for doing so: the proof of this partition given in
  \cite{burnolblocks} contained an erroneous argument.  We explain the error,
  show how the correct part of that argument does imply the partition, and
  elucidate the structure of chains of complements ordered by the prefix
  relation.
\item We then turn to the explanation on how one can compute numerically all
  terms of the modal series.  We present a scheme which will not need to
  compute explicitly the coefficients of the eigenpolynomials (Part~2
  explained how this can be done explicitly in finitely many steps, these
  coefficients being rational numbers).
\item We conclude with what happens if we use formally $k=0$ in the
  block-Irwin modal expansion Equation~\eqref{eq:modalIrwin}.  The answer,
  already presented but with no justification at the end of Part~2 in
  \cite{burnolmodal}, separates the cases $w$ bordered, $w$ unbordered with
  zero as leading digit, and $w$ unbordered with a non-zero first digit.
\end{enumerate}

\section{The Guibas-Odlyzko partition of the occurrence-free language}
\label{sec:Fpartition}

Equation~(224) of Part~2 states
\begin{equation}\label{eq:Fpartition}
  \cF = \cR \sqcup \bigsqcup_{c\in \cC_w^+} c\cR.
\end{equation}
This partition appears indirectly (in a mirror form)
in \cite[Proof of Th.\@ 2.1, equation~(2.2)]{guibasodlyzko1981b}.  It is also
tacit, for the binary radix, in Example~6.4 of \cite{odlyzko1995}, which again, as
\cite{guibasodlyzko1981b}, only states a counting function identity for the
cardinalities per length. One can obviously consider though that
Equation~\eqref{eq:Fpartition} is to be credited to \cite{guibasodlyzko1981a,
  guibasodlyzko1981b,
  odlyzko1995}, and this is how the author actually understood their arguments
when reading them.

The author mentioned this partition in \cite{burnolblocks} at the end of
section~2, but, alas, with a wrong proof: in the argument leading to
\cite[Eq.~(21)]{burnolblocks}, the sentence \emph{this characterizes $i$ as
  being the $j$-th positive period of $w$} is wrong.  Sadly, this wrong
premise was used to realize the partition via looking at the values of
$k_w(wg)$ for $g\in \cF$. But this is false, the Guibas-Odlyzko partition
Equation~\eqref{eq:Fpartition} is usually finer.  For example with $w=00100$,
which has overlap periods $3$ and $4$, with respective right-complements $100$
and $0100$, one finds $k_w(w100)= k_w(w0100) = 2$.  But the correct part in
\cite{burnolblocks} does imply the partition
Equation~\eqref{eq:Fpartition}. We explain this, and then elucidate entirely
the structure of chains of complements ordered by the prefix relation.

Recall that $k_w^+(g)$ is defined as $k_w(wg) -1$ for any word $g$.  The
correct part in \cite{burnolblocks} is the reasoning which we use here to
prove the following proposition, which is stated as Equation~(210) in
\cite{burnolmodal}, with the comment 
\emph{Several members of $\cC_w^+$ may be prefixes of the same $Y$. When they
  are ordered by length, the corresponding values of $k_w^+$ decrease
  successively by one}.
\begin{prop}\label{prop:bordertelescope}
  Let $y$ be a formal parameter.  For any word $Y$,
  \begin{equation}\label{eq:bordertelescope}
    y^{k_w(Y)}  =  y^{k_w^+(Y)}  +  (1-y)\sum_{\substack{c\in\cC_w^+\\Y=cZ}}  y^{k_w^+(Z)}.
  \end{equation}
\end{prop}
\begin{proof}[Start of the proof]
  Let $n$ be the number of \emph{crossing occurrences} of $w$ in $wY$, i.e.\@
  those not equal to the initial one and not being contained in $Y$, so
  $k_w^+(Y)=k_w(Y)+n$.  We will argue below that the multiset $\{k_w^+(Z) -
  k_w(Y)\mid Y = cZ, c\in \cC_w^+\}$ has no multiplicities and is empty if
  $n=0$ and else is exactly $\{0,\dots, n-1\}$ ($n-1$ is for the shortest
  complement $c$ such that $c$ is a prefix of $Y$, and so on...).  So the
  identity is simply
  \begin{equation}
    y^{k_w(Y)} = y^{n+k_w(Y)}+(1-y)\sum_{i=0}^{n-1}y^{i+k_w(Y)}\,.
  \end{equation}
\end{proof}
Using Equation~\eqref{eq:bordertelescope} with $y=0$ recovers the
partition~\eqref{eq:Fpartition}: as $k_w^+(Z)=k_w(wZ)-1$,
$k_w^+(Z)=0$ if and only if $Z\in \cR$.

We now provide the full details, and for this, we start with a review of some
notions related to periods and borders of the non-empty word $w = w_1\dots
w_p$. A \emph{period} $j$ is a non-negative integer less than $p$ such that
$w_{i+j}=w_i$ if $i+j\leq p$. So we consider $j=0$ to be a period. If $j$ is
positive we say that we have an \emph{overlap period}. If $j$ is a period, we
set $r_j=w_1\dots w_j$ and call it the left pattern. We let
$b_j=w_{j+1}\dots w_p =w_1\dots w_{p-j}$ and call it a border \emph{only if
  $j>0$}, and we let $s_j = w_{p-j+1}\dots w_p$ and call it the (on-the-right)
complement of the border.  So $w=r_jb_j=b_js_j$ and $ws_j = r_jw = r_jb_js_j$.
Conversely, an equation $xw=wy$ with $0<|x|<p$ makes $|x|$ an overlap period of $w$.
Notice that, for $j>0$, $s_j$ is a \emph{cyclic permutation} of $r_j$: its
first digit is $w_a$ with $1\leq a\leq j$ and $a\equiv p+1 \bmod{j}$.  So
$s_j$ is obtained via $p$ one-unit cyclic \emph{left}-rotations of $r_j$. For
$j=0$, we have $r_0=s_0=\emptyword$ and $b_0=w$.

We observe that if $j_1<j_2$ are two periods, then $r_{j_1}\prec r_{j_2}$
where $\prec$ means ``is a prefix of''. And $s_{j_1}$ is a \emph{suffix} of
$s_{j_2}$.  There is no a priori reason for $s_{j_1}$ to be a \emph{prefix} of
$s_{j_2}$. We gave the example above with $w=00100$, $s_3=100$, $s_4=0100$.

However, this happens in the following circumstances: consider, for any word
$g$, the \emph{crossing occurrences} in $wg$.  Suppose that their
number is positive and let $0<j_1<j_2<\dots <j_n$ be the shifts from the
initial occurrence of $w$ for these $n$ {crossing} occurrences, which start
(or are, in view of our set-theoretic definition) respectively at the indices $j_1+1$,
\dots, $j_n+1$; in particular $j_n<p$.  So, for each $i=1$, \dots, $n$, $j_i$
is an overlap period of $w$, and the first $j_i$ digits of $g$ constitute the
complement $s_{j_i}$.  Thus these complements, and the empty word, ordered
from shortest to longest, are a chain of prefixes:
\begin{equation}
  \emptyword \prec s_{j_1}\prec \dots \prec s_{j_n} \prec g.
\end{equation}
Take the shortest complement $s_{j_1}$, if there is one, and write $g =
s_{j_1}z_1$. Recall that $s_{j_1}$ is a suffix of $w$:
\begin{equation}
   w_{j_1+1}\dots w_p g =  b_{j_1}g = b_{j_1}s_{j_1}z_1 = wz_1\,.
\end{equation}
Consider an occurrence of $w$ in $wz_1=w_{j_1+1}\dots w_p g$ which is not the
initial one and is not contained in $g$. As it is not contained in $g$ it must
start at some $w_{j_1+q}$, $1\leq q\leq p-j_1$. As it is not the initial one $q>1$.
But this means exactly that $j_1+q$ is a starting index of a crossing occurrence of
$w$ in $wg$.  So there are $n-1$ of them:
\begin{equation}
  k_w^+(z_1) = n - 1 + k_w(g).
\end{equation}
A similar proof shows that, if $n\geq2$ and defining $z_2$ such that $g=s_{j_2}z_2$,
\begin{equation}
  k_w^+(z_2) = n - 2 + k_w(g),
\end{equation}
and so on until 
\begin{equation}
  k_w^+(z_n) = k_w(g).
\end{equation}
And of course $k_w^+(g) = n + k_w(g)$.  This completes the proof of
Proposition~\ref{prop:bordertelescope}.

The error in \cite{burnolblocks} near the end of section~2 is as simple as it
is appalling: the author's reasoning presupposed that (all) the complements
are completely ordered by the prefix relation, which is very wrong in
general. The confusion arose with the subset of complements which are prefixes
of a given word $g$.  Motivated by this blunder, we now prove the following
statement, and later give a complete description of the possible chains of
complements ordered by the prefix relation.
\begin{prop}\label{prop:prefixedcomplements}
  Let $w$ be a non-empty word and let $0<a<b$ be two overlap periods. Let
  $s_a$ be the suffix of $w$ of length $a$, and $s_b$ the one of length $b$.
  Then $s_a$ is a prefix of $s_b$ if and only if $\gcd(a,b)$ is an overlap
  period.  Then $s_a$ and $s_b$ are powers of the same block of length
  $\gcd(a,b)$.
\end{prop}
We give two proofs. The first one was the first found by the author, the
second one is more directly related to the
Lyndon-Schützenberger theorems \cite{lyndonschutz1962} (see also
\cite[\S1.5]{alloucheshallitCUP2003}) and Guibas-Odlyzko
\cite{guibasodlyzko1981a,guibasodlyzko1981b}.
\begin{proof}[First proof of Proposition~\ref{prop:prefixedcomplements}]
  We let as usual $p$ be the length $|w|$ of $w$.

  If $d=\gcd(a,b)$ is a period then, with $a=id$, $b=jd$, there holds $s_a=z^i$, $s_b=z^j$
  with $z$ the suffix of $w$ of length $d$.  So $s_a$ is a prefix of $s_b$.

  For the converse: as $s_b$ is a subword of $w$, and $a$ is a period of the
  latter, it is a period of $s_b$, and the prefix of length $a$ of $s_b$,
  which by hypothesis is $s_a$, is the cyclic permutation of $s_a$ by $b$
  units to the \emph{right}. Suppose first that $b$ is a multiple $qa$, then
  $d=\gcd(a,b)$ is indeed a period. Else let $0<u<a$ be such that $b=qa+u$,
  $q\geq1$. So $s_a$ is invariant under the cyclic shift by $u$ units to the
  right. The cyclic group in $\ZZ/a\ZZ$ generated by $u$ is, as is well-known,
  generated by $d=\gcd(a,b)$.  So $s_a$ is invariant under the cyclic shift by
  $d$ units to the right.  If $d=1$, this causes $s_a$ to be the repetition of
  a single digit $e$, hence $w=e^{p}$ (because the prefix of length $a$ is a
  cyclic permutation of $s_a=e^a$) and $d$ is a period. Suppose $d>1$ and let
  $z$ be the length $d$ suffix of $s_a$.  So $s_a = z^j$ for some $j\geq1$ as
  $a$ is a multiple of $d$. Write $s_b = xs_a^q$, with $x$ of length $u$. On
  one hand, $u$ is a multiple $kd$ of $d$, and on the other hand $x$ is the
  prefix of length $u$ of $s_a$, because $s_a$ is a prefix of $s_b$. So
  $x=z^k$, and $s_b = z^{jq+k}$.  Write $w=ys_a^m$ with
  $|y|<a$. If $y$ is empty, we are done. Else $y$ is a non-empty prefix of the
  prefix $r_a$ of $w$ of length $a$, which is a cyclic permutation of $s_a$,
  hence is invariant under the cyclic permutation either to the left or to the
  right by $d$ units.  So, for $1\leq i \leq a-d$, we are certain that
  $w_i=w_{i+d}$. We also know that $w_i=w_{i+d}$ if $i>|y|$ (and $i+d\leq p$)
  because then $w_i$ is located in the $s_a^m=z^{ mj}$ suffix of $w$. There
  remains the case with $a-d< i \leq |y|$. As $i<a$, $w_i$ is located in
  $r_a$, and $w_i = w_{i+d-a}$ with $1\leq i+d-a\leq a$. As $a$ is a period we
  get $w_i=w_{i+d}$ if $i+d\leq p$. So $d$ is a period of $w$.
\end{proof}

\begin{proof}[Second proof of Proposition~\ref{prop:prefixedcomplements}]
  This proof is slightly more sophisticated. Again, if $d=\gcd(a,b)$ is a
  period, both $s_a$ and $s_b$ are powers of $z$ where $z$ is the $w$-suffix
  of length $d$.  For the converse, let $r_a$ be the prefix of $w$ of length
  $a$ and $r_b$ the one of length $b$.  One has $w=r_bX=Xs_b$, so
  $r_bw=ws_b$. Now, write $r_b=r_ax$ and, by hypothesis, $s_b=s_ay$.  We
  obtain $r_axw=ws_ay = r_awy$.  Trimming $r_a$ gives $xw=wy$.  But, as it was
  reviewed before, this makes $0<b-a<p$ a period of $w$.  We now
  establish the following Lemma: \emph{if $t$ and $u$ are overlap periods and
    $t+u\leq p$, then $\gcd(t,u)$ is a period.}  Proof: this is true if $t=u$,
  so without loss of generality suppose $0<t<u$. We show that $u-t$ is a
  period. If the index $i$ is such that $i+u\leq p$, then
  $w_i=w_{i+u}=w_{i+u-t}$.  If $i>p-u\geq t$, we have
  $w_i=w_{i-t}=w_{i-t+u}$. So $u-t$ is a period. Now replace $(t,u)$ by
  $(t_1=\min(t,u-t),u_1=\max(t,u-t))$ and repeat. After finitely many steps we will
  be in the situation with $u_n=2t_n$, $u_n-t_n=t_n=\gcd(t,u)$. This completes
  the proof of the Lemma. Applying it with $t=a$ and $u=b-a$, we obtain that
  $\gcd(t,u)=\gcd(a,b)$ is a period.
\end{proof} 
There is a stronger \emph{GCD rule} in \cite{guibasodlyzko1981a} and
\cite{guibasodlyzko1981b}: if $t$ and $u$ are overlap periods and
\begin{equation}
  t+u< p +\gcd(t,u),
\end{equation}
then $\gcd(t,u)$ is a period.  In the second proof above, we only needed this
conclusion under the hypothesis $t+u\leq p$ and we gave the proof in that
case.  Let us call this the \emph{easy GCD rule}.  It has an immediate
consequence on the structure of the set of overlap periods.  Suppose $w$ is
bordered, of length $p$, and let $p_0$ be the shortest positive period.  Let
$j=\lceil p_0^{-1}p\rceil -1$, so that $j$ is the largest with $jp_0<p$. The
set of positive periods is $\{p_0, 2p_0,\dots, jp_0\} \cup E$ where the
exceptional set $E$ is included in the open interval $(p-p_0,p)$ and contains
no multiple of $p_0$.  The proof: suppose $0<q\leq p-p_0$ is a period. The
\emph{easy GCD rule} implies that $\gcd(p_0,q)$ is a period, so it must be
$p_0$, and $q$ has to be a multiple of $p_0$. So, a period $q$ which is not a
multiple of $p_0$ must verify $p-p_0<q<p$.
\begin{prop}
  Let $w$ be a bordered word of length $p$ with at least two periods.  Let
  $p_0$ be the shortest overlap period.  Let $g$ be an arbitrary word. Then one of the
  following mutually exclusive possibilities is realized:
  \begin{itemize}
  \item No (right)-complement of $w$ is a prefix of $g$.
  \item The complements of $w$ which are prefixes of $g$ are the suffixes of
    $w$ of lengths $p_0$, $2p_0$, \dots, $i_gp_0$, for some $i_g$ with $i_gp_0<p$.
  \item There is only one complement of $w$ which is a prefix of $g$, and its
    length is not a multiple of $p_0$.
  \end{itemize}
\end{prop}
\begin{proof}
  Suppose $s_1$ of length $a_1>0$ and $s_2$ of length $a_2>a_1$ are two
  complements of $w$ which are prefixes of $g$.  So $a_1$ and $a_2$ are
  periods and $s_1$ is a prefix of $s_2$. It follows from
  Proposition~\ref{prop:prefixedcomplements} that, setting $d=\gcd(a_1,a_2)$,
  it is a period and $s_1= z^i$, $s_2 = z^j$, with $z$ the suffix of $w$ of
  length $d$.  Suppose $p-p_0\leq a_1$. Then $0<a_2-a_1<p_0$ and is a period
  because it is a multiple of $d$ which is a period. This is impossible.  So
  $a_1<p-p_0$ which implies, by the remarks preceding the Proposition, that $a_1$ is
  a multiple of $p_0$ and $s_1$ is a power of $z_0$, with $z_0$ the
  length-$p_0$ suffix of $w$.  So the latter is a prefix of $s_2$. By
  Proposition~\ref{prop:prefixedcomplements} again, this is only possible if $p_0$
  divides $a_2$, and $s_2$ is then also a power of $z_0$.  Finally, if $z_0^i$ is the
  longest complement which arises as a prefix of $g$, all lower positive exponents
  are allowed, too.
\end{proof}

\section{An autonomous equation for \texorpdfstring{$\mu_0$ and the $\mu_k$, $k\geq1$}
                                    {mu0 and muk}}


The measures $\eta_k$ play a central role due to the renewal identity $\eta_k
= K_w^k\eta_0$ which provides the structural understanding of the dependence
of $I(b,w,k)$ on $k\geq1$.  We reduced on the word space the $y$-combined
$\wmu_k$ measures $\wM(y)$ to the $y$-combined $\weta_k$ measures $\wH(y)$ via
Equations~(208) and (212).  It will be useful to again use the free language
partition, as already embedded in (212), but now for eliminating $\wH(y)$ in
favor of $\wM(y)$.  Here is how it goes, working directly on Borel measures on
$\Iff01$.  Let
\begin{equation}
  \rM_y=\sum_{k\geq0}y^k\mu_k,\qquad\qquad
  \rH_y=\sum_{k\geq0}y^k\eta_k,
\end{equation}
be considered as formal power series with measure coefficients.
The identities referred to above become on $\Iff01$ are:
\begin{align}\label{eq:firstdigit}
  \Bigl(I-b^{-1}\sum_{a\in\Sigma_b}P_a\Bigr)\rM_y
  &=\delta_0-(1-y)b^{-p}P_w\rH_y
\\
  \rM_y&=\Bigl(I+(1-y)\sum_h b^{-h}P_{s_h}\Bigr)\rH_y\,.
\end{align}
Here and below, $h$ ranges over the overlap periods of $w$, and $P_s$
acts on Borel measures by push-forward under $\phi_s$.

Recall that $r_hw=ws_h$, with $r_h$ the left pattern and $s_h$ the
right complement. Consequently,
\begin{equation}
  \Bigl(I+(1-y)\sum_h b^{-h}P_{r_h}\Bigr)P_w
  =P_w\Bigl(I+(1-y)\sum_h b^{-h}P_{s_h}\Bigr).
\end{equation}
Multiplying the first-digit identity Equation~\eqref{eq:firstdigit} on the
left by $I+(1-y)\sum_h b^{-h}P_{r_h}$ therefore eliminates $\rH_y$ and gives
this proposition:
\begin{prop}[Autonomous equations for the measures $\mu_k$]
There holds
\begin{equation}\label{eq:autonomousM}
  \begin{split}
    \biggl(\Bigl(I+(1-y)\sum_h b^{-h}P_{r_h}\Bigr)
    \Bigl(I-b^{-1}\sum_{a\in\Sigma_b}P_a\Bigr)
    +(1-y)b^{-p}P_w\biggr)\rM_y
\\    =\delta_0+(1-y)\sum_h b^{-h}\delta_{x(r_h)}.
  \end{split}
\end{equation}
In particular, at $y=0$, we obtain
  \begin{equation}\label{eq:autonomousmu0}
    \biggl(\Bigl(I+\sum_h b^{-h}P_{r_h}\Bigr)
    \Bigl(I-b^{-1}\sum_{a\in\Sigma_b}P_a\Bigr)+b^{-p}P_w\biggr)\mu_0
    =\delta_0+\sum_h b^{-h}\delta_{x(r_h)}\,.
  \end{equation}
\end{prop}
We use this in the next section to compute the
moments of $\mu_0$.

Coefficient extraction gives an equation for each $\mu_k$ involving only
$\mu_k$ and $\mu_{k-1}$.  The explicit ``forcing term'' can be non-zero only
for $k=0,1$.  For completeness, extracting the coefficient of $y$ and using
Equation~\eqref{eq:autonomousmu0}, the identity relating $\mu_1$ with $\mu_0$
reads:
\begin{equation}
  \biggl(\Bigl(I+\sum_h b^{-h}P_{r_h}\Bigr)
  \Bigl(I-b^{-1}\sum_{a\in\Sigma_b}P_a\Bigr)+b^{-p}P_w\biggr)\mu_1
  =\delta_0-\Bigl(I-b^{-1}\sum_{a\in\Sigma_b}P_a\Bigr)\mu_0.
\end{equation}
For $k\geq2$, the relation between $\mu_k$ and $\mu_{k-1}$, which we omit
here, is homogeneous.

\section{Effective termwise computation of the modal series}

Let us first review how one can compute the moments of $\eta_0$. We recall the
notation $v_{0;m} = \int_{\Ifo01}x^m\deta_0$.  Also, as in the preceding
section, summations over the symbol $h$ mean summations over the overlap
periods of $w$, and $r_h$ is the left pattern, $s_h$ the right
complement. Following Equation~(411) of Part~2, put
$\sT_s(f)=b^{-|s|}f\circ\phi_s$. Recall that
\begin{equation}
  \sD_w=I-\sum_{a\in\Sigma_b}\sT_a+\sum_h\sT_{s_h}
  -\sum_h\sum_{a\in\Sigma_b}\sT_{as_h}+\sT_w,
\end{equation}
and
\begin{equation}
  \sE_w=\sum_h\sT_{s_h}
  -\sum_h\sum_{a\in\Sigma_b}\sT_{as_h}+\sT_w.
\end{equation}
The
coefficients $d_{N,i}$ and $e_{N,i}$ explicitly given in
Equations~(375)--(376) of Part~2 (where $t_m=b^{-m-1}$) are such that
\begin{equation}
  \sD_w(x^N) = \sum_{i=0}^N\binom{N}{i} d_{N,i}x^i,
  \qquad
  \sE_w(x^N) = \sum_{i=0}^N\binom{N}{i} e_{N,i}x^i.
\end{equation}
As $D_w\eta_0=\delta_0$, we obtain the system, for $m\geq1$,
\begin{equation}\label{eq:kernelvtriang}
  D(b^{-m-1})v_{0;m}
  =-\sum_{i=0}^{m-1}\binom{m}{i} d_{m,i}v_{0;i},\qquad v_{0;0}=b^p.
\end{equation}
As explained already in Part~2, $D(b^{-m-1})>0$, and this allows to
compute all needed $v_{0;m}$.

For $u_{0;m}=\int_{\Ifo01}x^m\dmu_0(x)$, we shall use Equation~\eqref{eq:autonomousmu0}.
Define, for $0\leq i\leq m$
\begin{equation}
  \begin{split}
  c_{m,i}= \delta_{mi}&-b^{-m-1}\sum_{a\in\Sigma_b}a^{m-i}
                         +b^{-(m+1)p}n(w)^{m-i}
  \\ &+\sum_h b^{-(m+1)h}n(r_h)^{m-i}
  -\sum_h\sum_{a\in\Sigma_b}b^{-(m+1)(h+1)}n(r_h a)^{m-i}.
  \end{split}
\end{equation}
As everywhere else the convention $0^0=1$ applies.  Integration against $x^m$
of Equation~\eqref{eq:autonomousmu0} gives, for $m\geq1$,
\begin{equation}
  D(b^{-m-1})u_{0;m}=\sum_h b^{-(m+1)h}n(r_h)^m
  -\sum_{i=0}^{m-1}\binom{m}{i}c_{m,i}u_{0;i}.
\end{equation}
And the recurrence is initialized by
\begin{equation}
  u_{0;0}=b^pC(b^{-1}).
\end{equation}

To compute the moments of $\nu_j$, put as before
$\sigma_j=(-1)^{j-1}\nu_j^{(j-1)}/(j-1)!$.
For any function with $j-1$ continuous derivatives, from Equation~(414)
there holds $\sigma_j(\sE_w f)=\lambda_j(w)\sigma_j(\sD_w f)$, which gives
\begin{equation}\label{eq:kernelvanishing}
  \sigma_j\circ\bigl(\sE_w-\lambda_j(w)\sD_w\bigr)=0.
\end{equation}
Applying this to $x^{j+n-1}$ gives, for $n\geq1$,
\begin{equation}\label{eq:nujmoments}
  \begin{split}
  &D(b^{-j-n})\bigl(\lambda_j(w)-\lambda_{j+n}(w)\bigr)\int_0^1x^n\dnu_j(x)
  \\ &\qquad =\sum_{q=0}^{n-1}\binom{n}{q}
  \bigl(e_{j+n-1,j+q-1}-\lambda_j(w)d_{j+n-1,j+q-1}\bigr)\int_0^1x^q\dnu_j(x).
  \end{split}
\end{equation}
The initial value is $\nu_j(\Iff01)=1$.

As the $\cB_k$ are monic of degree $k$, the polynomial $x^m$ is a unique
linear combination of $\cB_0,\dots,\cB_m$. By biorthogonality, its
coefficient at $\cB_k$ is $\sigma_{k+1}(x^m)$. The definition of
$\sigma_{k+1}$ gives, for $0\leq k\leq m$,
\begin{equation}
  \sigma_{k+1}(x^m)=\frac1{k!}\int_0^1 \frac{d^k}{dx^k}(x^m)\dnu_{k+1}(x)
  =\binom{m}{k}\int_0^1x^{m-k}\dnu_{k+1}(x).
\end{equation}
Hence the polynomial identity
\begin{equation}\label{eq:polid}
  x^m=\sum_{k=0}^m\binom{m}{k}\Bigl(\int_0^1t^{m-k}\dnu_{k+1}(t)\Bigr)\cB_k(x).
\end{equation}
Integrating against $\mu_0$ gives, for $m\geq1$,
\begin{equation}\label{eq:modalmu0triang}
  \int_{\Ifo01}\cB_m(x)\dmu_0(x)=u_{0;m}
  -\sum_{k=0}^{m-1}\binom{m}{k}\Bigl(\int_0^1x^{m-k}\dnu_{k+1}(x)\Bigr)
  \int_{\Ifo01}\cB_k(x)\dmu_0(x),
\end{equation}
with initial value $\int_{\Ifo01}\cB_0(x)\dmu_0(x)=b^pC(b^{-1})$.

Likewise, integrating Equation~\eqref{eq:polid} against $\eta_0$ gives
\begin{equation}\label{eq:kernelRtriang}
  \fR_{m+1}(b,w)=v_{0;m}-\sum_{k=0}^{m-1}\binom{m}{k}
  \Bigl(\int_0^1x^{m-k}\dnu_{k+1}(x)\Bigr)\fR_{k+1}(b,w),
\end{equation}
with initial value $\fR_1(b,w)=b^p$.  Recall that $\fR_{m+1}(b,w)=
\int_{\Ifo01}\cB_m(x)\deta_0(x)$.

We can compute numerically $A_j(n)$ for $n>1$ using a binomial series, which
gives a geometrically convergent expression:
\begin{equation}\label{eq:newton}
  A_j(n)=\int_0^1\frac{\dnu_j(x)}{(n+x)^j}
        = \sum_{q=0}^{\infty}(-1)^q\binom{j+q-1}{q}n^{-j-q}
                   \int_0^1x^q\dnu_j(x).
\end{equation}
There is also a level-raising identity for these higher order Stieltjes values.
Applying Equation~\eqref{eq:kernelvanishing} to $1/(n+x)$ gives
\begin{equation}\label{eq:kernelAjraising}
  \begin{split}
  A_j(n)={}&\sum_{a=0}^{b-1}A_j(bn+a)
  +\frac{1-\lambda_j(w)}{\lambda_j(w)}A_j(b^pn+n(w))
  \\ &+\frac{1-\lambda_j(w)}{\lambda_j(w)}\sum_h
  \Bigl(A_j(b^{h}n+n(s_h))-\sum_{a=0}^{b-1}A_j(b^{h+1}n+n(as_h))\Bigr).
  \end{split}
\end{equation}
Every argument on the right is at least $bn$. We apply this identity to
all the initial digits $n=1,\dots,b-1$ before using the binomial series.

Thus:
\begin{prop}
  All terms of the modal series Equation~\eqref{eq:modalKempner} can be
  effectively computed to arbitrary precision using, first, linear triangular
  systems of recurrences to compute (exactly as rational numbers if desired)
  the moments of $\mu_0$, those of $\eta_0$, and those of $\nu_j$ for as many
  $j$'s as are needed; then, level-raising for the higher order Stieltjes
  integrals and Newton binomial series allow to compute the desired results
  using arbitrarily fast geometrically convergent series for each term of the
  modal expansion.
\end{prop}

\section{The exceptional case \texorpdfstring{$w=0^p$}{w=0...0}}

The level raising of all the initial kernels has already given
Equation~\eqref{eq:modalKempnerzero}. Its terms are
\begin{equation}
  \begin{split}
  \beta_{m+1}^{(2)}(b,0^p)=(-1)^m
  \Bigl(&\sum_{a=b}^{b^2-1}A_{m+1}(a)\int_{\Ifo01}\cB_m(x)\dmu_0(x)
  \\&-\sum_{a=1}^{b-1}A_{m+1}(ab^p)\fR_{m+1}(b,0^p)\Bigr).
  \end{split}
\end{equation}
All arguments of the $A_j$ are at least $b$ and we can use for all the
expansions from Equation~\eqref{eq:newton}.  The moments of $\mu_0$, $\eta_0$,
and $\nu_j$, and the polynomial integrals, are computed by the same triangular
recurrences as in the preceding sections.

Here the overlap periods are $1,\dots,p-1$, with left patterns and
complements both equal to $0^h$. Hence
\begin{equation}
  C(t)=1+t+\dots+t^{p-1},\qquad
  D(t)=1-(b-1)(t+\dots+t^p).
\end{equation}
In particular, the first approximation, including the finite harmonic
sum kept outside the modal series, is
\begin{equation}
  \sum_{a=1}^{b-1}\frac1a+\beta_1^{(2)}(b,0^p)
  =b^pC(b^{-1})\log(b)
  +\sum_{a=1}^{b-1}\Bigl(\frac1a-b^p\log\Bigl(1+\frac1{ab^p}\Bigr)\Bigr).
\end{equation}

For $p=1$, one has $\mu_0=\eta_0$. The subtracted terms cancel exactly
the terms indexed by multiples of $b$ in the first sum, leaving
\begin{equation}
  \beta_{m+1}^{(2)}(b,0)
  =(-1)^m\Bigl(\sum_{a=1}^{b-1}\sum_{d=1}^{b-1}
  A_{m+1}(ba+d)\Bigr)\fR_{m+1}(b,0).
\end{equation}
The remaining kernels have arguments with two non-zero digits.

\section{Illustration with the modal series for binary \texorpdfstring{$01$}{01}}

Let us consider the example with $b=2$ and $w=01$, which is unbordered and begins
with zero.  A binary word avoids $01$ if and only if it is of the
form $1^q0^r$, with $q\geq0$ and $r\geq0$.  Its numerical value is
\begin{equation}
  n(1^q0^r)=(2^q-1)2^r.
\end{equation}
The Kempner sum itself in this example has the form
\begin{equation}
  I(2,01,0)  =  \sum_{q\geq1}\sum_{r\geq0} 
  \frac1{(2^q-1)2^r}  =  2\sum_{q\geq1}\frac1{2^q-1}  
                      =  3.2133903048305835\dots.
\end{equation}
It is one of the few cases of Kempner sums which can be computed
directly by brute force.  And that case is even more favourable, as
the quantity above is twice $I(2,0,0)$, which is the Erd\H{o}s--Borwein
constant \cite{borwein}.  Clausen's identity \cite{clausen1828} gives the more
rapidly convergent expression
\begin{equation}
  I(2,01,0)  =  2\sum_{q\geq1} \frac{2^{-q^2}(2^q+1)}{2^q-1}.
\end{equation}
Thus we have an independent computation against which to check both the modal
series, and the numerical algorithm from \cite[Part~2]{burnolmodal}, to high
precision.

As the word is unbordered, $\mu_0=\eta_0$, and
\begin{equation}
  D(t)=(1-t)^2,\qquad E(t)=t^2,\qquad
  \lambda_j(01)=\frac1{(2^j-1)^2}.
\end{equation}
Equation~\eqref{eq:modalKempner} becomes
\begin{equation}
  I(2,01,0)=\sum_{j=1}^\infty(-1)^{j-1}A_j(1)\fR_j(2,01).
\end{equation}
In particular,
\begin{equation}\label{eq:01firstmode}
  \beta_1(2,01)=4\log2=2.7725887222397812\dots.
\end{equation}

\begin{table}[htbp]
  \caption{Modal coefficients for $I(2,01,0)$}
  \label{tab:1}
  \centering
    \begin{tabular}{r@{\qquad}r@{\qquad}r}
      $j$ & $\beta_j(2,01)$ & $\beta_j(2,01)/\beta_{j-1}(2,01)$ \\
      \hline
      $1$ & $2.7725887222397812$ & --- \\
      $2$ & $0.3599098141053758$ & $0.1298$ \\
      $3$ & $0.0633812159125540$ & $0.1761$ \\
      $4$ & $0.0134118886516573$ & $0.2116$ \\
      $5$ & $0.0031038476755685$ & $0.2314$ \\
      $6$ & $0.0007492714617026$ & $0.2414$ \\
      $7$ & $0.0001844731790820$ & $0.2462$ \\
      $8$ & $0.0000458306028737$ & $0.2484$ \\
      $9$ & $0.0000114321253699$ & $0.2494$ \\
      $10$ & $0.0000028565081479$ & $0.2499$ \\
    \end{tabular}
\end{table}

The ordinary moments $v_{0;m}=\int_{\Ifo01}x^m\deta_0(x)$ start with
$v_{0;0}=4$. The triangular recurrence Equation~\eqref{eq:kernelvtriang}
specializes, for $m\geq1$, to
\begin{equation}
  (1-2^{-m-1})^2v_{0;m}
  =(2^{-m-1}-2^{-2m-2})\sum_{i=0}^{m-1}\binom{m}{i} v_{0;i}.
\end{equation}
The first few are:
\begin{equation}
  v_{0;0}=4,\qquad v_{0;1}=\frac43\;,\qquad
  v_{0;2}=\frac{20}{21}\;,\qquad v_{0;3}=\frac{76}{105}\;.
\end{equation}
For the moments of $\nu_j$, write
$M_{j;n}=\int_0^1x^n\dnu_j(x)$. The general recurrence
Equation~\eqref{eq:nujmoments} specializes to, for $n\geq1$,
\begin{equation}
  (2^n-1)(2^{j+n}+2^j-2)M_{j;n}=(2^n+2^j-2)\sum_{q=0}^{n-1}\binom{n}{q} M_{j;q},
\end{equation}
with initial value $M_{j;0}=1$, for each $j\geq1$.
In particular, $M_{j;1}=2^j/(3\cdot2^j-2)$.
For $j=1$, these moments are $M_{1;n}=1/(n+1)$,
as expected from $\nu_1=\Leb$.

We can now compute the quantities $\fR_{m+1}(2,01) =
\int_{\Ifo01}\cB_m(x)\deta_0(x)$ without having to get the coefficients of the
eigenpolynomials.  Indeed they verify the triangular recurrence
Equation~\eqref{eq:kernelRtriang}
\begin{equation}
  \fR_{m+1}(2,01)=v_{0;m}-\sum_{k=0}^{m-1}\binom{m}{k} M_{k+1;m-k}\fR_{k+1}(2,01),
\end{equation}
with initial value $\fR_1(2,01)=4$.

It remains to compute the $A_j(1)$. With $\lambda_j=\lambda_j(01)$,
the level-raising identity specializes to
\begin{equation}
  A_j(n)=A_j(2n)+A_j(2n+1)+(\lambda_j^{-1}-1)A_j(4n+1).
\end{equation}
In particular,
\begin{equation}
  A_j(1)=A_j(2)+A_j(3)+(\lambda_j^{-1}-1)A_j(5).
\end{equation}
All values on the right become geometrically convergent
series using the Newton expansion Equation~\eqref{eq:newton}.
For $j=1$, $\nu_1=\Leb$ and $A_1(N)=\log(1+1/N)$, so the first
modal term requires no moment computation. For the subsequent terms,
the moments of the $\nu_j$ are computed from Equation~\eqref{eq:nujmoments}.
The moments $v_{0;m}$ are reused for every $j$.

The first ten modal coefficients are given in Table~\ref{tab:1}.
They are truncated to sixteen decimal places;
the ratios are rounded to four decimal places.

\section{The modal series of Part~2 if used with \texorpdfstring{$p=1$}{p=1}
  is that of Part~1}

In this section $w=d$ is a single digit. We express the left and right modal
coefficients of Part~2 in terms of the notation of Part~1.  Here
$\cC_d^+=\emptyset$, $D(t)=1-(b-1)t$, $E(t)=t$, and
\begin{equation}
  \lambda_j(d)  = \frac{E(b^{-j})}{D(b^{-j})} = \frac{b^{-j}}{1-b^{-j+1}+b^{-j}}  
                = \frac{1}{b^j-b+1}\;,
\end{equation}
is indeed the eigenvalue from Part~1. This is not surprising as the return
operator of Part~2 specializes, of course, to the one of Part~1. 
The eigenpolynomials $\cB_{j-1}$ and
the probability measures $\nu_j$ are the same in the two parts. 
Also note that $\eta_k=\mu_k$ for every $k\geq0$.

The right modal coefficient is defined in Equation~(343) of Part~2.  Equation
(168) of Part~1, taken with $k=0$, therefore gives
\begin{equation}
  \fR_j(b,d)  = \int_{\Ifo01}\cB_{j-1}(x)\dmu_0(x)  = a_j(b,d)\lambda_j.
\end{equation}
Equation~(181) of Part~1 then yields
\begin{equation}
  \fR_j(b,d)=(-1)^{j-1}b\,c_j(b,d)\lambda_j.
\end{equation}
For $\fL_j(b,d)$, Equation~(432) of Part~2 simplifies to
\begin{equation}
  \fL_j(b,d) = (-1)^{j-1}
  \Bigl(
    \Un_{\{d>0\}}(d)A_j(d)+ \lambda_j\sum_{\substack{1\leq a<b\\a\neq d}}A_j(a)
  \Bigr),
\end{equation}
where $A_j(n)$ is from Equation~(170) of Part~1:
\begin{equation}
  A_j(n) = \int_0^1 \frac{\dnu_j(x)}{(n+x)^j}\;.
\end{equation}
By the first-digit self-similarity of $\nu_j$, applied to the decomposition
of
\begin{equation}
  F_j(b,d)=\int_{b^{-1}}^1\frac{\dnu_j(x)}{x^j}\;
\end{equation}
over the successive intervals of length $b^{-1}$,
\begin{equation}
  F_j(b,d)  =  \Un_{\{d>0\}}(d)(b^{j}-b+1)A_j(d)  +  \sum_{\substack{1\leq a<b\\a\neq d}}A_j(a).
\end{equation}
Hence
\begin{equation}
  \fL_j(b,d)=(-1)^{j-1}\lambda_jF_j(b,d).
\end{equation}
It follows that, for every $k$,
\begin{equation}
  \fL_j(b,d)\lambda_j^{k-1}\fR_j(b,d)
  =b\,c_j(b,d)F_j(b,d)\lambda_j^{k+1}\,.
\end{equation}
This shows in particular how the $k\geq1$ formulas of Theorems~2 in Part~1 and
3 in Part~2 match.

If $d=0$, we know from Theorem~2 of Part~1, that
\begin{equation}\label{eq:d0k0}
  b^{-1}I(b,0,0) = \sum_{j=1}^{\infty}c_j(b,0)F_j(b,0)\lambda_j\,.
\end{equation}
So the modal formula of Part~2 also holds as written with $k=0$ in that case $p=1$, $w=d=0$.

\section{Formal \texorpdfstring{$k=0$}{k=0}: unbordered words with a leading zero}

Let us consider generally an unbordered word $w$ of length $p\geq1$.
Then
$D(t)=1-bt+t^p$, $E(t)=t^p$, and
\begin{equation}
  \lambda_j(w) = \frac{b^{-pj}}{1-b^{1-j}+b^{-pj}}.
\end{equation}
The shortest next-return word is $w$ itself, and
\begin{equation}
  x_w=\frac{n(w)}{b^p-1}.
\end{equation}
Equation~(432) of Part~2 simplifies to
\begin{equation}\label{eq:Ljnoborders}
  \fL_j(b,w)  =  (-1)^{j-1}
  \Bigl(\lambda_j(w)\sum_{a=1}^{b-1}A_j(a) - 
                    \bigl(\lambda_j(w)-1\bigr)
                     \Un_{\{w_1\neq0\}}(w)A_j(n(w))  \Bigr).
\end{equation}

Consider the case with the first digit $w_1=0$.
Equation~\eqref{eq:Ljnoborders} gives
\begin{equation}
  \lambda_j(w)^{-1}\fL_j(b,w)  =  (-1)^{j-1}\sum_{a=1}^{b-1}A_j(a).
\end{equation}
Equation~(362) of Part~2 reduces to
\begin{equation}
  I(b,w,0)=\sum_{a=1}^{b-1}V_0(a), \qquad
  \textrm{with }V_0(n)=\int_{\Ifo01}\frac{\deta_0(x)}{n+x}\;.
\end{equation}
We can assume here $p>1$ as
the $p=1$ case was handled in the previous section. As $w$ is unbordered, it
is not $0^p$, thus $x_w>0$, and all Stieltjes kernels
$1/(n+z)$, $n\geq1$, can be expanded in the modal basis of eigenpolynomials
as per Theorem~4 of \cite{burnolmodal}.
This gives, with absolute convergence,
\begin{equation}
  I(b,w,0)  = 
   \sum_{j=1}^{\infty}  (-1)^{j-1}  \left(\sum_{a=1}^{b-1}A_j(a)\right)  \fR_j(b,w).
\end{equation}
Hence
\begin{equation}\label{eq:k0formal}
  I(b,w,0)  =  \sum_{j=1}^{\infty}  \fL_j(b,w)\lambda_j(w)^{-1}\fR_j(b,w).
\end{equation}
We had reached the same conclusion for $p=1$, $w=d=0$: using formally
$k=0$ in the $k\geq1$ modal series gives an absolutely convergent
representation of $I(b,w,0)$.

For $p>1$, Equations~\eqref{eq:k0formal},
and \eqref{eq:modalKempner} are identical term by term.

\section{Formal \texorpdfstring{$k=0$}{k=0}: unbordered words with a non-zero leading digit}

Suppose that $w$ is an unbordered word with $w_1\neq0$.
We show that the formal use of
the $k\geq1$ modal series for $k=0$ leads to a divergent
series. Equation~(362) of Part~2 gives
\begin{equation}
  I(b,w,0)  =  \sum_{a=1}^{b-1}V_0(a)-V_0(n(w)).
\end{equation}
Observe that we are in a situation with $x_w>0$, so we can expand
in the modal basis each involved kernel $1/(n+z)$, and, using again the
notation from \cite{burnolmodal}
\begin{equation}
  A_j(N) = \int_{0}^1 \frac{\dnu_j(x)}{(N+x)^j}\;,
\end{equation}
we get
\begin{equation}\label{eq:k0unborderedexact}
  I(b,w,0)  =  \sum_{j=1}^{\infty}  (-1)^{j-1}  
    \Bigl( \sum_{a=1}^{b-1}A_j(a)-A_j(n(w))  \Bigr)  \fR_j(b,w),
\end{equation}
with absolute convergence.

On the other hand, division of Equation~\eqref{eq:Ljnoborders} by
$\lambda_j(w)$ shows that the difference between the $j$-th term of the
formal $k=0$ substitution and the $j$-th term of
Equation~\eqref{eq:k0unborderedexact} is
\begin{equation}\label{eq:unborderedextra}
  (-1)^{j-1}  \frac{A_j(n(w))}{\lambda_j(w)}  \fR_j(b,w).
\end{equation}
The quantity $A_j(n(w))$ is the result of applying on the analytic function
$1/(n(w)+z)$, considered in the disk centered at $x_w$ and of radius $R=b^p
x_w=n(w)+x_w\geq1+x_w>\max(x_w,1-x_w)$ (as $n(w)+x_w = b^p x_w$), the
compactly supported distribution $\nu_j^{(j-1)}/(j-1)!$, also denoted
$(-1)^{j-1}\sigma_j$, where, explicitly,
$\sigma_j(f)=\frac{1}{(j-1)!}\int_0^1 f^{(j-1)}(x)\dnu_j(x)$ for any $f\in
C^{j-1}(\Iff01)$.  The geometric series
\begin{equation}\label{eq:geoseries}
  \frac1{n(w)+z}
  = \sum_{m\geq0}(-1)^m R^{-1}\frac{(z-x_w)^m}{R^m}
\end{equation}
converges in the topology of smooth functions on $\Iff01$, so we can apply
termwise the functional $\sigma_j$.  This functional vanishes on all powers
with exponents $<j-1$.  Equations~(320)--(326) of Part~2 show that
\begin{equation}
  \sum_{m\geq j}
  \Bigl|R^{j-1}\sigma_j\bigl(R^{-m}(z-x_w)^m\bigr)\Bigr|=O_{j\to\infty}(\theta_1^j)
\end{equation}
for some $0<\theta_1<1$.  The term $m=j-1$ obtained from Equation~\eqref{eq:geoseries}
after applying $(-1)^{j-1}R^{j-1}\sigma_j$ is $R^{-1}$.  Therefore
\begin{equation}\label{eq:ajasymp}
  (b^p x_w)^{j-1}A_j(n(w))\to_{j\to\infty} \frac1{b^p x_w}\;.
\end{equation}

Let us now study $\fR_j(b,w)$ as $j\to\infty$.  By Equation~(343) of Part~2,
\begin{equation}
  \fR_j(b,w) = \int_{\Ifo01}\cB_{j-1}(x)\deta_0(x) 
  = \sum_{r\in \cR} b^{-|r|}\cB_{j-1}(x(r)).
\end{equation}
Here we are using the tail language $\cR$ and we recall that in the prefix
algebra, $D_w^{-1} = \sum_{r\in \cR} b^{-|r|}P_{r}$, hence on (continuous on
$\Iff01$) functions $(\sD_w^{-1}f)(x) = \sum_{r\in \cR} b^{-|r|}f(\phi_r(x))$,
and in particular,
\begin{equation}
  \fR_j(b,w) = (\sD_w^{-1}\cB_{j-1})(0).
\end{equation}
As $w$ is unbordered, Equation~(412) of Part~2 gives $\sE_w=\sT_w$ (with
$\sT_w(f) = b^{-|w|}f\circ \phi_w$, see Equation~(411)).  From Equation~(413)
one has $\sK_w = \sE_w\sD_w^{-1}$, so the eigenvalue equation gives
\begin{equation}
  b^{-|w|}\bigl(\sD_w^{-1}\cB_{j-1})\circ \phi_w = \lambda_j(w)\cB_{j-1}.
\end{equation}
Now this computation can be done entirely in the finite dimensional space of
polynomials of degrees at most $j-1$.  Evaluating at $-n(w)$, whose image
under $\phi_w$ is $0$, we get
\begin{equation}
  b^{-p} \bigl(\sD_w^{-1}\cB_{j-1}\bigr)(0) = \lambda_j(w)\cB_{j-1}(-n(w)),
\end{equation}
and thus
\begin{equation}
  \fR_j(b,w) = b^p \lambda_j(w)\cB_{j-1}(-n(w)).
\end{equation}

Set as before $R=n(w)+x_w=b^px_w$.
Proposition~10 of Part~2 gives, for some $0<\theta_2<1$,
\begin{equation}
  \Bigl\| R^{-j+1}\cB_{j-1}  - \Bigl(\frac{z-x_w}{R}\Bigr)^{j-1} \Bigr\|_{\cH_R} 
  = O(\theta_2^j).
\end{equation}
Set
\begin{equation}
  Q_j(z)  =  R^{-j+1}\cB_{j-1}(z)  -  \left(\frac{z-x_w}{R}\right)^{j-1}.
\end{equation}
Write
\begin{equation}
  Q_j(z)  =  \sum_{m=0}^{j-2}q_{j,m}  \left(\frac{z-x_w}{R}\right)^m.
\end{equation}
The displayed monomials form the orthonormal basis of $\cH_R$, hence
\begin{equation}
  \|Q_j\|_{\cH_R}^2
  =
  \sum_{m=0}^{j-2}|q_{j,m}|^2.
\end{equation}
Evaluating at $z=-n(w)$ gives $Q_j(-n(w)) =  \sum_{m=0}^{j-2}q_{j,m}(-1)^m$, so 
\begin{equation}
  |Q_j(-n(w))| \leq \sqrt{j-1}\|Q_j\|_{\cH_R} = O(\sqrt{j}\theta_2^j).
\end{equation}
It follows that
\begin{equation}\label{eq:bjasymp}
  \cB_{j-1}(-n(w))  =  (-1)^{j-1}R^{j-1}(1+o(1)).
\end{equation}
On the other hand
\begin{equation}
  b^p\lambda_j(w)R^{j-1}  =  \frac{x_w^{j-1}}{1-b^{1-j}+b^{-pj}}  =  x_w^{j-1}(1+o(1)).
\end{equation}
Hence
\begin{equation}
  \fR_j(b,w)  =  (-1)^{j-1}x_w^{j-1}(1+o(1)).
\end{equation}
We obtain finally
\begin{equation}\label{eq:limit}
  (-1)^{j-1}  \frac{A_j(n(w))}{\lambda_j(w)}  \fR_j(b,w)  \longrightarrow  \frac1{x_w}\;,
\end{equation}
and we can at long last assert that the modal series for $k\geq1$, if used
formally at $k=0$ with $w$ an unbordered word such that $w_1>0$, diverges.

\section{Formal \texorpdfstring{$k=0$}{k=0}: bordered words and language factorizations}

In this section, $w$ is bordered with longest border $v$.

We use subscripts $\cL_v$, $\cG_v$, $\cR_v$, $\cF_v$ for languages associated
with $v$.  We also use such subscripts for $w$ itself. 

Let $g$ be any word
containing at least one occurrence of $v$.  Cut $g$ into $rs$ with $r$ ending
right after the first occurrence of $v$ (our language is a bit ambiguous
because we have called ``occurrence'' the starting index; here of course we
mean that $v$ must be included in $r$ as a suffix). We observe the useful
counting (or rather trimming) lemma:
\begin{equation}\label{eq:trimming}
  k_w(g) = k_w^+(s) = k_w(ws) - 1.
\end{equation}
Indeed, $w=xv$ with $x=w_1\dots w_{p_0}$ and $p_0$ the smallest overlap period of
$w$. So, in the word $ws=xvs$, an occurrence of $w$ starting at any $w_i$ with
$1<i\leq p_0$ would cause an overlap of $w$ with itself longer than
$|v|$. This is excluded, so the starting location must be in $vs$. Either this
starting location is in $s$, or it is in $v$ which means that it creates also
a crossing occurrence in $r\mid s$ (where $\mid$ is the boundary). Conversely
a crossing occurrence in $r\mid s$ cannot start strictly to the left of the
$v$-suffix of $r$ (else the terminal $v$ would not be the first one in $r$). So
there are as many crossing occurrences in $g=r\mid s$ as there are in $w\mid
s$. Thus $k_w(g)$ is the number of crossing occurrences in $w\mid s$ plus the
number of occurrences in $s$. This is what the above equation says.

Suppose now that $g\in \cF_w\setminus \cF_v$. Splitting as above $g$ as $rs$,
we get $k_w^+(s)=0$, which is the definition of the tail language $\cR_w$:
$s\in \cR_w$. Conversely, if $r\in\cL_v$ and $s\in \cR_w$, then $k_w(rs)=
k_w^+(s)=0$ and $rs\in \cF_w$.  Thus:
\begin{equation}
  \cF_w = \cF_v \sqcup \cL_v\cR_w\,.
\end{equation}
Restricting to those words starting with a non-zero digit
\begin{equation}
  \cF_w^+ = \cF_v^+ \sqcup \cL_v^+ \cR_w\,.
\end{equation}
In turn, this implies
\begin{equation}\label{eq:wminusv}
  \begin{split}
    I(b,w,0) - I(b,v,0) &= \sum_{g\in\cL_v^+, r\in \cR_w}\frac1{n(g)b^{|r|} + n(r)}
\\
    &= \sum_{r\in \cR_w} b^{-|r|}f_{\cL_v}(x(r)) = \int_{\Ifo{0}{1}}f_{\cL_v}(x)\deta_0(x).
  \end{split}
\end{equation}
Let $m_v=\min_{g\in\cL_v^+}n(g)$.
The poles of $f_{\cL_v}$ are located at the $-n(g)$, $g\in\cL_v^+$, hence
at distance at least $m_v + x_w$ to $x_w$.
If $|v|\geq2$, then every $g\in\cL_v^+$ has length at least $2$,
and $n(g)\geq b$.
If $v=0$, any $g\in\cL_v^+$ has length at least $2$ and again $n(g)\geq b$.
If $v=d>0$ then $n(g)\geq d\geq1$,
and as $v$ is a border of $w$, $w\neq 0^p$ and $x_w>0$, so
$1+x_w>R_w = \max(x_w,1-x_w)$.
Thus in all cases one can choose a radius $R$ such that
\begin{equation}
  R_w<R<m_v+x_w\,,
\end{equation}
and consider the associated Hardy space and Riesz basis of normalized
eigenpolynomials:  the block-directed Euler--Maclaurin expansion associated with
$w$ applies to $f_{\cL_v}$, exactly as we did earlier with $f_{\cL_w}$.  We
therefore write
\begin{equation}\label{eq:fLv}
  f_{\cL_v}(z)  =  \sum_{j=1}^{\infty}a_j\cB_{j-1}(z),
\end{equation}
where
\begin{equation}
  a_j  =  \frac1{(j-1)!} \int_0^1 f_{\cL_v}^{(j-1)}(x)\dnu_j(x).
\end{equation}
Applying now integration against the measure $\eta_0$, which is associated
with the tail language $\cR_w$, and recalling the definition of the $w$-right
modal coefficients $\fR_j(b,w) = \int_{\Ifo01}\cB_{j-1}(x)\deta_0(x)$, we
get from Equation~\eqref{eq:wminusv}, an absolutely convergent series
\begin{equation}\label{eq:wminusv2}
  I(b,w,0)-I(b,v,0)
  =
  \sum_{j=1}^{\infty}a_j\fR_j(b,w).
\end{equation}

It remains to identify the coefficients $a_j$, and for this we identify a
second factorization related to the language $\cL_v$ (it can be motivated from
looking at the associated generating function of counts per word lengths, but
we go directly to the result):
\begin{equation}\label{eq:Lfac}
  \cL_w=\cL_v\cG_w.
\end{equation}
Indeed, let $g\in\cL_w$, and split $g=rs$ immediately after the first
occurrence of $v$ so that $r\in\cL_v$.  Since $g\in\cL_w$, $k_w(g)=1$ and $w$
is terminal in $g$.  Applying the trimming lemma Equation~\eqref{eq:trimming}
first to $g=rs$, we get $k_w^+(s)=1$. Removing the last digit both in $g$ and
in $s$ (which is possible because $g=r$ is excluded by $g\in \cL_w$) gives
$k_w^+(s[{:}{-}1])=0$.  Thus $s\in\cG_w$.  Conversely, let $r\in\cL_v$ and
$s\in\cG_w$.  We now obtain again from the trimming lemma at start of this
section $k_w(rs)=k_w^+(s)=1$ and $k_w((rs)[{:}{-}1])=k_w^+(s[{:}{-}1])=0$ and
$rs\in\cL_w$.  The factorization is uniquely determined from the cut being
made immediately after the first occurrence of $v$.  Restricting to words with
non-zero first digit, Equation~\eqref{eq:Lfac} gives
\begin{equation}
  \cL_w^+=\cL_v^+\cG_w.
\end{equation}

Consequently, using $n(rg)=b^{|g|}n(r)+n(g)$, we obtain
\begin{equation}\label{eq:Kfac}
  f_{\cL_w}(z)
  =\sum_{g\in\cG_w}b^{-|g|}f_{\cL_v}(\phi_g(z))
  =(\sK_w f_{\cL_v})(z).
\end{equation}
Apply $\sK_w$ to both sides of Equation~\eqref{eq:fLv}, and use the eigenvalue
relation
\begin{equation}
  \sK_w\cB_{j-1}  =  \lambda_j(w)\cB_{j-1}.
\end{equation}
This gives
\begin{equation}
  f_{\cL_w}  =  \sum_{j=1}^{\infty}  a_j\lambda_j(w)\cB_{j-1}.
\end{equation}
The expansion of $f_{\cL_w}$ used for the $k\geq1$ modal
Equation~\eqref{eq:modalIrwin} is
\begin{equation}
  f_{\cL_w}  =  \sum_{j=1}^{\infty}  \fL_j(b,w)\cB_{j-1}.
\end{equation}
Uniqueness therefore gives
\begin{equation}\label{eq:aj}
  a_j=\lambda_j(w)^{-1}\fL_j(b,w),
\end{equation}
and Equation~\eqref{eq:wminusv2} becomes finally
\begin{equation}
  I(b,w,0)-I(b,v,0)  =  \sum_{j=1}^{\infty}  \fL_j(b,w)\lambda_j(w)^{-1}\fR_j(b,w),
\end{equation}
with absolute convergence.
We summarize the last three sections in one last theorem:
\begin{theo}
  Using formally $k=0$ in the modal series Equation~\eqref{eq:modalIrwin}
  which is valid for $k\geq1$ gives:
  \begin{itemize}
  \item a divergent series if $w$ is unbordered and its first digit is not zero,
  \item an absolutely convergent series computing $I(b,w,0)$ if $w$ is unbordered and its first digit is zero,
  \item an absolutely convergent series computing $I(b,w,0)-I(b,v,0)$ if $w$
    is bordered and $v$ is its longest border.
  \end{itemize}
\end{theo}

\bigskip
\noindent
  Université de Lille,
  Faculté des Sciences et technologies,
  Département de mathématiques,
  Cité Scientifique,
  F-59655 Villeneuve d'Ascq cedex,
  France.
\newline
\strut \texttt{jean-francois.burnol@univ-lille.fr}

\end{document}